\documentclass{amsart}

\usepackage{hyperref}

\usepackage{amssymb}
\usepackage{amsmath}
\usepackage{amsthm}
\usepackage{enumerate}
\usepackage{graphicx}
\usepackage{tikz-cd}

\newcommand{\ds}{\displaystyle}

\newcommand{\mult}{\operatorname{mult}}

\theoremstyle{plain}
\newtheorem{theorem}{Theorem}

\newtheorem{lemma}[theorem]{Lemma}

\theoremstyle{definition}

\newtheorem*{example}{Example}

\theoremstyle{remark}

\begin{document}

\title{Computing polynomial conformal models for low-degree Blaschke products}

\date{January 19, 2018}

\author[T. Richards]{Trevor Richards}
\address{Department of Mathematics and Computer Science, Monmouth College, Monmouth, IL 61462, USA.}
\email{trichards@monmouthcollege.edu}
\author[M. Younsi]{Malik Younsi}
\thanks{Second author supported by NSF Grant DMS-1758295}
\address{Department of Mathematics, University of Hawaii Manoa, Honolulu, HI 96822, USA.}
\email{malik.younsi@gmail.com}

\keywords{Conformal models, Blaschke products, polynomials, fingerprints.}
\subjclass[2010]{primary 30J10; secondary 30C35, 30C10.}

\begin{abstract}
For any finite Blaschke product $B$, there is an injective analytic map $\varphi:\mathbb{D}\to\mathbb{C}$ and a polynomial $p$ of the same degree as $B$ such that $B=p\circ\varphi$ on $\mathbb{D}$. Several proofs of this result have been given over the past several years, using fundamentally different methods. However, even for low-degree Blaschke products, no method has hitherto been developed to explicitly compute the polynomial $p$ or the associated conformal map $\varphi$.  In this paper, we show how these functions may be computed for a Blaschke product of degree at most three, as well as for Blaschke products of arbitrary degree whose zeros are equally spaced on a circle centered at the origin.
\end{abstract}

\maketitle

\section{Introduction}

For domains $D,E\subset\mathbb{C}$ and analytic functions $f:D\to\mathbb{C}$, $g:E\to\mathbb{C}$, we say that $g$ on $E$ is a conformal model for $f$ on $D$ if there is some analytic bijection $\varphi:D\to E$ such that $f=g\circ\varphi$ on $D$.  By precomposing both sides of this equation by $\varphi^{-1}:E\to D$, it follows immediately that $f$ on $D$ is a conformal model for $g$ on $E$.  In this case, we say that the pairs $(f,D)$ and $(g,E)$ are conformally equivalent, and it is easy to see that this defines an equivalence relation on the set of pairs of the form $(f,D)$.

There has been significant interest in recent years in the ``Polynomial Conformal Modeling Question'' (PCMQ), which asks whether a given pair $(f,D)$ has any conformal model $(g,E)$ for which the function $g$ is a polynomial.  A positive answer is known for the PCMQ when the domain $D$ is the unit disk $\mathbb{D}$ and the function $f$ is a finite Blaschke product.  In this case, the polynomial conformal model may be assumed to have the same degree as $f$. In the following, we use the notation $\mathbb{D}_p=\{z\in\mathbb{C}:|p(z)|<1\}$ for any polynomial $p\in\mathbb{C}[z]$.

\begin{theorem}\label{thm: PCMQ for finite Blaschke products.}
For any finite Blaschke product $B$, there is a polynomial $p$ of the same degree as $B$ such that $p$ on $\mathbb{D}_p$ is a conformal model for $B$ on $\mathbb{D}$.
\end{theorem}

Theorem~\ref{thm: PCMQ for finite Blaschke products.} has seen several proofs in recent years using a variety of approaches. We mention

\begin{itemize}
    \item the characterization of fingerprints of polynomial lemniscates obtained by Ebenfelt et. al.~\cite{EKS} in view of applications to computer vision, which has Theorem~\ref{thm: PCMQ for finite Blaschke products.} as a corollary;
    \item the proof of the first author \cite{R1} using critical level curve configurations;
    \item the proof of the second author \cite{Y} using conformal welding.

\end{itemize}

Theorem~\ref{thm: PCMQ for finite Blaschke products.} was further generalized in \cite{R2} to arbitrary functions analytic on the closed unit disk. As far as we know, the most general version is a solution on the online mathematics forum \textit{math.stackexchange.com} by Lowther and Speyer~\cite{L,S} showing that the PCMQ has a positive answer as long as the domain $D$ is bounded and the function $f$ is analytic on the closure of $D$. This solution relies on approximation by polynomials interpolating certain derivative data, and can readily be generalized to meromorphic functions $f$, in which case polynomials need to be replaced by rational maps. Also of interest on this topic is the paper of the authors~\cite{RY} which again brings the tools of conformal welding to bear on the PCMQ, also addressing the question of the degree of the polynomial conformal model in more detail.

Conspicuously absent in the aforementioned proofs of Theorem~\ref{thm: PCMQ for finite Blaschke products.} is anything of a constructive nature. In his Master's Thesis \cite{FB}, Maxime Fortier Bourque asked whether there are explicit formulas relating the Blaschke product $B$ and its polynomial conformal model $p$. In this paper, we present a method for computing a polynomial conformal model as well as obtaining an algebraic formula for its associated conformal map in two cases : first for finite Blaschke products of degree at most three, in Section~\ref{sect: Degree three and below.}, and then for finite Blaschke products of arbitrary degree whose zeros are equally spaced on a circle centered at the origin, in Section~\ref{sect: Uniformly distributed zeros.}.

The work in Sections~\ref{sect: Degree three and below.}~and~\ref{sect: Uniformly distributed zeros.} requires several lemmas, which are proved in Section~\ref{sect: Proofs of lemmas.}.

\section{The Polynomial Conformal Model for a finite Blaschke product of degree at most three}
\label{sect: Degree three and below.}

First, observe that if $B$ is a degree one finite Blaschke product, then $B$ itself is injective on $\mathbb{D}$, so that we may write $B$ as $B=p\circ\varphi$, where the polynomial is $p(z)=z$ and the conformal map is $\varphi(z)=B(z)$.

If $B$ has degree two, then by precomposing $B$ with a disk automorphism sending $0$ to the critical point of $B$, we may assume without loss of generality that the two zeros of $B$ are symmetric with respect to the origin (this follows from Lemma~\ref{lem: Zeros uniform about origin.}). In other words, we may assume that $B$ falls in the case treated in Section~\ref{sect: Uniformly distributed zeros.}.

Consequently, it only remains to treat the case of a finite Blaschke product $B$ with $\deg(B)=3$. In this case, the derivative of $B$ is a rational function with at most four zeros, two of which lie in $\mathbb{D}$. The critical points of $B$ may hence be computed by means of the quartic formula (see for example~\cite{N1}).

Let $z_1$ and $z_2$ be the two critical points of $B$ in $\mathbb{D}$, and set $k_1=B(z_1)$ and $k_2=B(z_2)$. Note that if $p$ on $\mathbb{D}_p$ is a conformal model for $B$ on $\mathbb{D}$, then $p$ must have $k_1$ and $k_2$ as critical values as well. The converse also holds, as shown in the following lemma.

\begin{lemma}
\label{lem: Same critical values.}
Let $B$ be a degree three Blaschke product whose critical points in $\mathbb{D}$ are $z_1$ and $z_2$. If $p\in\mathbb{C}[z]$ is any degree three polynomial whose critical values are $B(z_1)$ and $B(z_2)$, then $p$ on $\mathbb{D}_p$ is a conformal model for $B$ on $\mathbb{D}$.
\end{lemma}

\begin{proof}
By Theorem~\ref{thm: PCMQ for finite Blaschke products.}, there is a degree three polynomial $\widehat{p}$ such that $\widehat{p}$ on $\mathbb{D}_{\widehat{p}}$ is a conformal model for $B$ on $\mathbb{D}$. As previously mentioned, it follows that the critical values of $\widehat{p}$ are precisely $B(z_1)$ and $B(z_2)$. On the other hand, there is exactly one degree three polynomial with any two given critical values, modulo precomposition with a linear map (see \cite{BCN}). Since $p$ and $\widehat{p}$ have the same critical values, it immediately follows that $(p,\mathbb{D}_p)$ is conformally equivalent to $(\widehat{p},\mathbb{D}_{\widehat{p}})$.  Finally, by the transitivity of conformal equivalence, we conclude that $p$ on $\mathbb{D}_p$ is a conformal model for $B$ on $\mathbb{D}$.

\end{proof}

Having proved Lemma~\ref{lem: Same critical values.}, it remains to show how to compute a degree three polynomial with two prescribed critical values.

If $p$ is the desired polynomial, then for any $\alpha,\beta\in\mathbb{C}$ with $\alpha\neq 0$, the polynomial $\widehat{p}(z)=p(\alpha z+\beta)$ has the same critical values as $p$.

Note that $\widehat{p}$ on $\mathbb{D}_{\widehat{p}}$ is a conformal model for $p$ on $\mathbb{D}_p$, hence is also a conformal model for $B$ on $\mathbb{D}$, again by transitivity. Moreover, one can easily check that by making an appropriate choice of $\alpha$ and $\beta$, we may choose $\widehat{p}$ to be of the form $$\widehat{p}(z)=z^3+cz+d$$ for some $c,d\in\mathbb{C}$. Let us simply replace $p$ with this new conformal model, using the letter $p$ to denote $\widehat{p}$. We now wish to compute $c$ and $d$ from the prescribed critical values of $p$, namely $k_1$ and $k_2$.

The critical points of $p$ are the two square roots of $\dfrac{-c}{3}$. Let $\pm z_1$ denote these two roots.  If $+z_1=-z_1$, then $c=0$, in which case $p(z)=z^3+d$. The derivative of such a polynomial has a zero of order two, so the same must be true for $B$ as well. By precomposing $B$ with the appropriate disk automorphism, we may assume that the double critical point of $B$ is at the origin, so that by Lemma~\ref{lem: Zeros uniform about origin.}, the zeros of $B$ are equally spaced on a circle centered at the origin. This case being treated in Section~\ref{sect: Uniformly distributed zeros.}, we shall henceforth assume that $+z_1\neq-z_1$.

Substituting these roots back into $p$ and setting the result equal to $k_1$ and $k_2$ respectively, we obtain the equations $$k_1={z_1}^3+cz_1+d\text{ and }k_2=(-z_1)^3+c(-z_1)+d.$$  The solutions $c$ and $d$ to this system of equations are easily found to be $$c=-3\left(\dfrac{k_2-k_1}{4}\right)^{2/3}\text{ and }d=\dfrac{k_1+k_2}{2},$$ for any choice of the third root in the equation for $c$.  In order to simplify the notation, we leave this as simply $c$ and $d$ in what follows, keeping in mind that these quantities are computed in terms of $k_1$ and $k_2$. With these values of $c$ and $d$, the polynomial $p(z)=z^3+cz+d$ has critical values $k_1$ and $k_2$ and therefore is a conformal model for $B$ on $\mathbb{D}$, by Lemma~\ref{lem: Same critical values.}

Having found the polynomial conformal model $p$, we can now obtain a formula for the corresponding conformal map $\varphi:\mathbb{D}\to \mathbb{D}_p$ satisfying $B=p\circ\varphi$. In order to do so, we treat the equation $B=p\circ\varphi$ as a polynomial equation in the variable $\varphi$, with coefficients in the ring of rational functions in $z$. In other words, the function $\varphi$ we are looking for is a solution to the equation $$0=\varphi^3+c\varphi+(d-B).$$  The cubic formula (see again~\cite{N1}) now implies that $\varphi$ has the form
$$\varphi=U+V,$$
where $$U=\sqrt[3]{-\dfrac{d-B}{2}+\sqrt{\dfrac{(d-B)^2}{4}+\dfrac{c^3}{27}}}\text{ and }V=\sqrt[3]{-\dfrac{d-B}{2}-\sqrt{\dfrac{(d-B)^2}{4}+\dfrac{c^3}{27}}}.$$  In the above formulas, the same choice must be made for the square roots in $U$ and in $V$, while the cubic roots must and can be chosen to ensure that $UV=\dfrac{-c}{3}$.

These constraints still leave three possible solutions $\varphi$, corresponding to the three possible choices for the cubic roots in the expressions for $U$ and $V$. Each of these solutions satisfies the equation $B=p\circ\varphi$ in $\mathbb{D}$, although in general only one is analytic in the disk. Indeed, in order to see this, suppose that the two critical points of $B$ are distinct, which is the case of interest, and let $\varphi_1,\varphi_2:\mathbb{D}\to\mathbb{C}$ be two analytic maps satisfying $B=p\circ\varphi_1=p\circ\varphi_2$. According to Lemma~\ref{lem: Analytic implies injective.}, both $\varphi_1$ and $\varphi_2$ are injective on $\mathbb{D}$, and $\varphi_1(\mathbb{D})=\varphi_2(\mathbb{D})=\mathbb{D}_p$.

Now, we have $B\circ{\varphi_1}^{-1}=B\circ{\varphi_2}^{-1}$ on $\mathbb{D}_p$, which implies that $\psi={\varphi_2}^{-1}\circ\varphi_1$ is a disk automorphism satisfying $B=B\circ\psi$ on $\mathbb{D}$. But then $\psi$ must be the identity, in view of Lemma~\ref{lem: Automorphism of disk is identity.}, so that $\varphi_1=\varphi_2$ on $\mathbb{D}$. We conclude that at most one of the three solutions to the equation
$$0=\varphi^3+c\varphi+(d-B)$$
will be analytic.

In practice, determining which of the three choices of $\varphi$ is analytic may be quite difficult, mostly due to the complicated nature of the formulas involved. We illustrate this by the following example.

\begin{figure}[!ht]
\begin{minipage}[b]{0.45\linewidth}
\centering
\includegraphics[width=\textwidth]{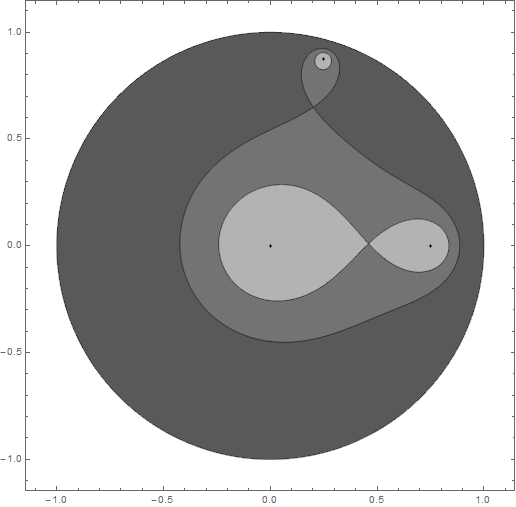}
\caption{The critical level curves of $B$.}
\label{fig: Level curves of B.}
\end{minipage}
\hspace{0.5cm}
\begin{minipage}[b]{0.45\linewidth}
\centering
\includegraphics[width=\textwidth]{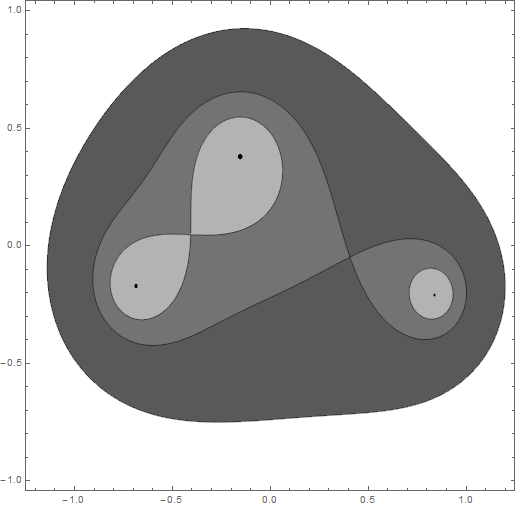}
\caption{The critical level curves of $p$.}
\label{fig: Level curves of p.}
\end{minipage}
\end{figure}

\begin{example}
Consider the finite Blaschke product $$B(z)=z\dfrac{(z-3/4)(z-(1/4+7i/8))}{(1-(3/4)z)(1-(1/4-7i/8)z)},$$ which has zeros at $0 ,1/2$ and $1/4+7i/8$. The derivative of $B$ is a degree four rational function, whose zeros can in principle be computed using the quartic formula. The closed forms for these critical points are too involved to display here, so we simply mention that those in the disk are approximately $z_1=0.2014+0.6494i$ and $z_2=0.4599+0.0103i$. The critical level curves of $B$ are displayed in Figure~\ref{fig: Level curves of B.}. The outer boundary of the shaded region is the unit circle.

The corresponding critical values are $k_1=B(z_1)$ and $k_2=B(z_2)$. Recall that by Lemma \ref{lem: Same critical values.}, there is a unique conformal equivalence class of degree three polynomials with two prescribed critical values, and any degree three polynomial $p$ having $k_1$ and $k_2$ as critical values is a polynomial conformal model on $\mathbb{D}_p$ for $B$ on $\mathbb{D}$.

The previous work shows that the polynomial $p(z)=z^3+cz+d$ (where $c=-3\left(\frac{k_2-k_1}{4}\right)^{2/3}$ and $d=\frac{k_1+k_2}{2}$) has critical values $k_1$ and $k_2$, so that $p$ on $\mathbb{D}_p$ is a conformal model for $B$ on $\mathbb{D}$.  Furthermore, as discussed above, the conformal map $\varphi:\mathbb{D}\to\mathbb{D}_p$ satisfying the equation $B=p\circ\varphi$ is defined by one of the three algebraic formulas obtained by applying the cubic formulas to the equation $B=p\circ\varphi$ written in the form $$\varphi^3+c\varphi+(d-B)=0$$ (again viewed as an equation in the unknown $\varphi$, with coefficients in the ring of rational functions in $z$).  The complexity of these formulas makes it difficult to determine precisely which of the three solutions is conformal on the unit disk (and therefore is truly the one we are looking for).

For comparison sake, we also display the critical level curves of the polynomial conformal model $p$ in Figure~\ref{fig: Level curves of p.}. The outer boundary of the shaded region is the set $\{z:|p(z)|=1\}$. The function $\varphi$ maps the lightest regions in Figure~\ref{fig: Level curves of B.} to the corresponding lightest regions in Figure~\ref{fig: Level curves of p.}, and so forth.

\end{example}

\section{The Polynomial Conformal Model for a finite Blaschke product with equally spaced zeros}
\label{sect: Uniformly distributed zeros.}

Let $B$ denote a degree $n\geq2$ Blaschke product whose zeros are equally spaced on a circle centered at $0$, i.e.
$$B(z)=\lambda\dfrac{z^n-c^n}{1-\bar{c}^nz^n}$$
for some $\lambda\in\mathbb{C}$ with $|\lambda|=1$ and some $c \in \mathbb{D}$.

Define $\varphi(z)=\dfrac{e^{i\pi/n}z}{\sqrt[n]{1-\bar{c}^nz^n}}$.  By Lemma~\ref{lem: varphi is injective analytic.}, the map $\varphi$ is analytic and injective on $\mathbb{D}$. A straightforward calculation shows that for $z\in\mathbb{D}$, $$B(\varphi(z))=\lambda\left(|c|^{2n}-1\right)z^n-\lambda c^n.$$
It follows that the polynomial $p(z)=\lambda(|c|^{2n}-1)z^n-\lambda c^n$ on the set $\varphi(\mathbb{D})$ is a polynomial conformal model for $B$ on $\mathbb{D}$. 

We remark that the formula for $p$ also appears in \cite[Section 4.6.2]{FB}, as mentioned to the authors by Maxime Fortier Bourque.

\section{Proofs of the lemmas}
\label{sect: Proofs of lemmas.}

\begin{lemma}\label{lem: Zeros uniform about origin.}
If the critical points of a finite Blaschke product $B$ of degree $n \geq 2$ are all at the origin, then the zeros of $B$ are equally spaced on a circle centered at the origin.
\end{lemma}

\begin{proof}

Let $\widehat{B}(z)=z^n$. Then $B$ and $\widehat{B}$ have the same critical points in $\mathbb{D}$, so there exists a disk automorphism $\tau: \mathbb{D} \to \mathbb{D}$ such that $B=\tau \circ \widehat{B}$ (see e.g. \cite{ZAK}). This implies that $B(z)=0$ if and only if $z^n=\tau^{-1}(0)$, so that the zeros of $B$ are indeed equally spaced on a circle centered at $0$.

\end{proof}

\begin{lemma}\label{lem: Analytic implies injective.}
Let $B$ be a finite Blaschke product and let $p$ be a polynomial with critical values all in $\mathbb{D}$. If $\deg(B)=\deg(p)$, then any analytic function $\varphi:\mathbb{D}\to\mathbb{C}$ such that $B=p\circ\varphi$ on $\mathbb{D}$ is injective on $\mathbb{D}$ and satisfies $\varphi(\mathbb{D})=\mathbb{D}_p$.
\end{lemma}

\begin{proof}
First note that since all the critical values of $p$ lie in $\mathbb{D}$, the set $\mathbb{D}_p$ is connected (see e.g. \cite[Proposition 2.1]{EKS}.) Also, we have $p \circ \varphi(\mathbb{D}) = B(\mathbb{D}) \subset \mathbb{D}$, from which we deduce that $\varphi(\mathbb{D}) \subset \mathbb{D}_p$. In order to show that $\varphi(\mathbb{D})=\mathbb{D}_p$, it suffices to prove that $\varphi(\mathbb{D})$ is both open and closed in $\mathbb{D}_p$, by connectedness. Clearly $\varphi(\mathbb{D})$ is open in $\mathbb{D}_p$, by the open mapping theorem. Now, suppose that $\{w_n\}\subset\varphi(\mathbb{D})$ is such that $w_n \to w \in \mathbb{D}_p$. Then $|p(w)|<1$, so $|p(w_n)| \not\to 1$, by continuity. For each $n$, let $z_n \in \mathbb{D}$ be chosen such that $w_n=\varphi(z_n)$. Passing to a subsequence if necessary, assume that $z_n \to z \in \overline{\mathbb{D}}$. Note that $$|B(z)|=\lim_{n \to \infty} |B(z_n)|= \lim_{n \to \infty} |p(\varphi(z_n))| = \lim_{n \to \infty} |p(w_n)|<1,$$ which shows that $|B(z)|<1$ and thus $|z|<1$. By continuity, we get $w=\varphi(z)$, so that $w \in \varphi(\mathbb{D})$. This shows that $\varphi(\mathbb{D})$ is closed in $\mathbb{D}_p$, so that $\varphi(\mathbb{D})=\mathbb{D}_p$.

Suppose now that $\varphi$ fails to be injective. Then for some distinct points $z_0,z_1\in\mathbb{D}$, $\varphi(z_0)=\varphi(z_1)$. Let $w \in \mathbb{D}$ denote the image under $p$ of this common value, so that $w=p(\varphi(z_0))=p(\varphi(z_1))$. Let $\zeta_1,\zeta_2,\ldots,\zeta_k$ denote the preimages of $w$ under $p$, ordered so that $\varphi(z_1)=\zeta_1$.  Since $\varphi(\mathbb{D})=\mathbb{D}_p$, for each $2\leq j\leq k$, there exists $z_j\in\mathbb{D}$ such that $\varphi(z_j)=\zeta_j$. Note that then $B(z_j)=p(\varphi(z_j))=p(\zeta_j)=w$ for each $0 \leq j \leq k$.

For any function $f$ which is analytic at a point $\xi_0\in\mathbb{C}$, we denote by $\mult_f(\xi_0)$ the multiplicity of $\xi_0$ as a solution to the equation $f(z)=f(\xi_0)$ (note of course that $\mult_f(\xi_0)\geq1$). We shall use the well-known fact that multiplicity is multiplicative : if $f$ is analytic at $\xi_0$, and $g$ is analytic at $f(\xi_0)$, then $\mult_{g\circ f}(\xi_0)=\mult_f(\xi_0)\cdot\mult_g(f(\xi_0))$. In particular, if $h=g\circ f$, then $\mult_{h}(\xi_0)\geq\mult_g(f(\xi_0))$.

Now, note that $\deg(B)=\ds\sum_{z\in B^{-1}(w)}\mult_B(z)\geq\sum_{j=0}^k\mult_B(z_j)$.  Writing $B=p\circ\varphi$ therefore yields $$\deg(B)\geq\ds\sum_{j=0}^k\mult_{p\circ\varphi}(z_j)=\ds\sum_{j=0}^k\mult_{\varphi}(z_j)\mult_p(\varphi(z_j))>\sum_{j=1}^k\mult_p(\zeta_j)=\deg(p).$$  This contradicts the assumption that $\deg(p)=\deg(B)$. It follows that $\varphi$ is injective, which completes the proof of the lemma.

\end{proof}

\begin{lemma}\label{lem: varphi is injective analytic.}
For any $c\in\mathbb{D}$, the map $\varphi(z)=\dfrac{z}{\sqrt[n]{1-\bar{c}^nz^n}}$ is an injective analytic map on $\mathbb{D}$.
\end{lemma}

\begin{proof}

Clearly, the function $\varphi$ is analytic on the unit disk, since $1-\overline{c}^nz^n$ is non-vanishing there.

Now suppose that $z,w \in \mathbb{D}$ are such that $\varphi(z)=\varphi(w)$. Then we have
$$z^n(1-\overline{c}^nw^n) = w^n(1-\overline{c}^nz^n)$$
so that $z^n=w^n$, and thus $z=e^{2\pi i k/n}w$ for some $k \in \{0,1,\dots,n-1\}$. Substituting back into the equation $\varphi(z)=\varphi(w)$ yields $e^{2\pi i k/n}w=w$, so that either $w=0$, in which case $z=0$, or $e^{2\pi i k/n}=1$. In both cases, we get $z=w$. It follows that $\varphi$ is injective.

\end{proof}

\begin{lemma}\label{lem: Automorphism of disk is identity.}
If $B$ is a degree three Blaschke product whose two critical points in $\mathbb{D}$ are distinct, and if $\psi:\mathbb{D}\to\mathbb{D}$ is a disk automorphism satisfying
\begin{equation}
\label{eq1}
B\circ\psi=B
\end{equation}
on $\mathbb{D}$, then $\psi$ is the identity map.
\end{lemma}

\begin{proof}
It easily follows from Schwarz's lemma that an analytic function from $\mathbb{D}$ into $\mathbb{D}$ with two distinct fixed points is the identity. It thus suffices to find two distinct points in $\mathbb{D}$ which are fixed by $\psi$.

Let $z_1,z_2\in\mathbb{D}$ denote the two distinct critical points of $B$.  We claim that $B(z_1)\neq B(z_2)$. Indeed, if not, then the rational function $B(z)-B(z_1)$ would have zeros of order at least two at both $z_1$ and $z_2$, which is impossible since $B(z)-B(z_1)$ is a degree three rational function. \footnote{The authors would like to thank user mercio from the online mathematics forum \textit{math.stackexchange.com} for suggesting this argument (though in the context of polynomials).}

Now, Equation (\ref{eq1}) combined with the chain rule implies that $\psi$ preserves the set of critical points of $B$. If $\psi(z_1)=z_2$, then again by Equation (\ref{eq1}) we would have $B(z_1)=B(\psi(z_1))=B(z_2)$, a contradiction. It follows that $\psi$ fixes the distinct points $z_1$ and $z_2$, as required.

\end{proof}

\bibliographystyle{amsplain}

\providecommand{\bysame}{\leavevmode\hbox to3em{\hrulefill}\thinspace}
\providecommand{\MR}{\relax\ifhmode\unskip\space\fi MR }
\providecommand{\MRhref}[2]{%
  \href{http://www.ams.org/mathscinet-getitem?mr=#1}{#2}
}
\providecommand{\href}[2]{#2}
\begin{thebibliography}{}

\end{thebibliography}


\begin{thebibliography}{99}

\bibitem{BCN}
A.F. Beardon, T.K. Carne and T.W. Ng,
The critical values of a polynomial,
\textsl{Constr. Approx.},
\textbf{18} (2002),
343--354.

\bibitem{EKS}
P. Ebenfelt, D. Khavinson and H.S. Shapiro,
Two-dimensional shapes and lemniscates,
\textsl{Contemp. Math.},
\textbf{553} (2011),
45--59.

\bibitem{FB}
M. Fortier Bourque,
\textsl{Applications quasiconformes et soudure conforme},
Masters Thesis, Universit\'e Laval,
2010.

\bibitem{L}
G. Lowther,
Conjecture: Every analytic function on the closed disk is conformally a polynomial,
\textsl{Mathematics Stack Exchange},
http://math.stackexchange.com/q/443351.

\bibitem{N1}
S. Neumark,
\textsl{Solution of Cubic and Quartic Equations},
Pergamon, Oxford,
1965.

\bibitem{R1}
T. Richards,
Level curve configurations and conformal equivalence of meromorphic functions,
\textsl{Comput. Methods. Funct. Theory},
\textbf{15} (2015),
323--371.

\bibitem{R2}
T. Richards,
Conformal equivalence of analytic functions on compact sets,
\textsl{Comput. Methods. Funct. Theory},
\textbf{16} (2016),
585--608.

\bibitem{RY}
T. Richards and M. Younsi,
Conformal models and fingerprints of pseudo-lemniscates,
\textsl{Constr. Approx.},
\textbf{45} (2017),
129--141.

\bibitem{S}
D. Speyer,
Conjecture: Every analytic function on the closed disk is conformally a polynomial,
\textsl{Mathematics Stack Exchange},
http://math.stackexchange.com/q/443351.

\bibitem{Y}
M. Younsi,
Shapes, fingerprints and rational lemniscates,
\textsl{Proc. Amer. Math. Soc.},
\textbf{144} (2015),
1087--1093.

\bibitem{ZAK}
S. Zakeri,
On critical points of proper holomorphic maps on the unit disk,
\textsl{Bull. London Math. Soc.},
\textbf{30} (1988),
62--66



\end{thebibliography}

\end{document}